\documentclass[12pt,a4paper]{amsart}
\makeatletter
\renewcommand\normalsize{%
    \@setfontsize\normalsize{11.7}{14pt plus .3pt minus .3pt}%
    \abovedisplayskip 10\p@ \@plus4\p@ \@minus4\p@
    \abovedisplayshortskip 6\p@ \@plus2\p@
    \belowdisplayshortskip 6\p@ \@plus2\p@
    \belowdisplayskip \abovedisplayskip}
\renewcommand\small{%
    \@setfontsize\small{9.5}{12\p@ plus .2\p@ minus .2\p@}%
    \abovedisplayskip 8.5\p@ \@plus4\p@ \@minus1\p@
    \belowdisplayskip \abovedisplayskip
    \abovedisplayshortskip \abovedisplayskip
    \belowdisplayshortskip \abovedisplayskip}
\renewcommand\footnotesize{%
    \@setfontsize\footnotesize{8.5}{9.25\p@ plus .1pt minus .1pt}
    \abovedisplayskip 6\p@ \@plus4\p@ \@minus1\p@
    \belowdisplayskip \abovedisplayskip
    \abovedisplayshortskip \abovedisplayskip
    \belowdisplayshortskip \abovedisplayskip}
\ifdefined\pdfpagewidth
\else
\fi
\calclayout
\makeatother

\usepackage[backend=bibtex, 
sortcites=true,
firstinits=true,
hyperref,
maxbibnames=99,
]{biblatex}

\bibliography{Bibliography}{}

\usepackage{enumerate}

\usepackage[centertags]{amsmath}
\usepackage{amsfonts}
\usepackage{amssymb}
\usepackage{amsthm}
\usepackage{newlfont}
\usepackage{mathtools}
\usepackage{tikz}

\theoremstyle{plain}
\newtheorem{theorem}{Theorem}[section]

\newtheorem{corollary}[theorem]{Corollary}
\newtheorem{lemma}[theorem]{Lemma}

\theoremstyle{definition}
\newtheorem{definition}[theorem]{Definition}

\newtheorem*{remark}{Remark}

\newcommand{\R}{\mathbb{R}}
\newcommand{\N}{\mathbb{N}}

\newcommand{\Deg}{\operatorname{Deg}}

\newcommand{\diam}{\operatorname{diam}}

\newcommand{\supp}{\operatorname{supp}}
\newcommand{\argmin}{\operatorname{argmin}}
\newcommand{\Lip}{\operatorname{Lip}}
\newcommand{\capa}{\operatorname{cap}}
\newcommand{\cl}{\operatorname{cl}}
\newcommand{\Ric}{\operatorname{Ric}}

\newcommand{\diamess}{\operatorname{diam}_{\operatorname{obs}}}

\newcommand{\D}{\Deg_{\max}}
\newcommand{\dc}{d_{\operatorname{comb}}}

\newcommand{\eps}{\varepsilon}
\newcommand{\as}{\alpha_{\operatorname{spectral}}}
\newcommand{\amod}{\alpha_{\operatorname{mod}}}
\newcommand{\Uint}{U_{\operatorname{int}}}
\newcommand{\Hidden}[1]{}

\DeclareMathOperator{\Ent}{Ent}

\usepackage{booktabs}
\usepackage{hyperref}

\begin{document}

\title[Ollivier curvature, isoperimetry, and Log-Soblev inequality]{Ollivier curvature, Isoperimetry, concentration, and Log-Sobolev inequalitiy}
\author{
Florentin M\"unch
}
\date{\today}
\maketitle

\begin{abstract}
We introduce a Laplacian separation principle for the the eikonal equation on Markov chains. As application, we prove an isoperimetric concentration inequality for Markov chains with non-negative Ollivier curvature. That is, every single point from the concentration profile yields an estimate for every point of the isoperimetric estimate. Applying to exponential and Gaussian concentration, we obtain affirmative answers to two open quesions by Erbar and Fathi.
Moreover, we prove that the modified log-Sobolev constant is at least the minimal Ollivier Ricci curvature, assuming non-negative Ollivier sectional curvature, i.e., the Ollivier Ricci curvature when replacing the $\ell_1$ by the $\ell_\infty$ Wasserstein distance. This settles a recent open Problem by Pedrotti.
We give a simple example showing that non-negative Ollivier sectional curvature is necessary to obtain a modified log-Sobolev inequality via positive Ollivier Ricci bound. 
This provides a counterexample to a conjecture by Peres and Tetali.
\end{abstract}

\tableofcontents


\section{Introduction}

In this article, we answer a variety of open questions regarding isoperimetric, functional and concentration inequalities on Markov chains with non-negative Ricci curvature:

\begin{itemize}
\item Peres,Tetali: Modified Log-Sobolev under positive Ollivier curvature
\item Ollivier, Problem M: Log Sobolev under some curvature bound 
\item Ollivier, Problem P: Sectional curvature using $\ell_\infty$ Wasserstein distance
\item Pedrotti:  Modified Log-Sobolev using $\ell_\infty$ Wasserstein curvature
\item Erbar, Fathi: Spectral gap via exponential concentration
\item Erbar, Fathi: Modified log-Sobolev via Gaussian concentration
\end{itemize}
For the first question, we give a counterexample, and all other questions are answered affirmatively.
Discrete Ricci curvature has sparked remarkable interest in numerous mathematical communities:

\begin{itemize}
\item Markov chain mixing, functional inequalities and cutoff \cite{salez2023cutoff,fathi2015curvature,
eldan2017transport,caputo2009convex,riekert2022convergence}

\item Multi-particle systems such as Glauber dynamics \cite{blanca2022mixing,erbar2017ricci,paulin2016mixing,
villemonais2020lower,holmes2014curvature}

\item Geometric group theory \cite{berestycki2014cutoff,siconolfi2020coxeter,
keisling2021medium,bar2022conjugation,
taback2023conjugation}

\item Discrete topology \cite{kempton2021homology,
saucan2017network, ni2015ricci,knill2014coloring,forman2003bochner}

\item Quantum relativity \cite{loll2019quantum,tee2021canonical,tee2021enhanced,
gorard2020some,loll2022quantum} 

\item Machine learning and neural networks \cite{anand2022topological,wee2021ollivier,
topping2021understanding,li2022curvature,ye2019curvature} 

\item Data analysis \cite{weber2016forman,sia2019ollivier,
sandhu2015analytical,samal2018comparative,
boguna2021network}

\end{itemize}

Indeed, various non-equivalent notions of discrete Ricci curvature have been introduced such as Forman curvature for cell complexes \cite{forman2003bochner} based on a Bochner-Weizenböck formula, Coarse Ricci curvature known as Ollivier curvature based on optimal transport \cite{ollivier2009ricci,lin2011ricci}, Bakry Emery curvature based on a Bochner formula \cite{schmuckenschlager1998curvature,lin2010ricci}, and entropic curvature based on convexity of the entropy \cite{erbar2012ricci,mielke2013geodesic}.
None of these curvatures coincide except Forman and Ollivier curvature when choosing the 2-cells to maximize the Forman curvature, see e.g. \cite{jost2019Liouville} for examples of different Ollivier and Bakry Emery curvature.
Among all curvature notions, Ollivier curvature seems the most popular due to its simplicity and its elegant geometric interpretation. Specifically, the Ollivier curvature is bounded from below by $K \in \R$ iff
\[
\Lip(P_t f) \leq e^{-Kt} \Lip(f),
\]
see \cite{munch2017ollivier} where $P_t = e^{\Delta t}$ is the heat semigroup.
One remarkable implicit feature of the Lipschitz decay is that the underlying distance for the Lipschitz constant can be chosen independently of the Markov chain.
Indeed, this observation seems to give rise to a completely unexplored branch of Riemannian geometry:
Given a (potentially weighted) manifold $M$ with two Riemannian metrics $g_1,g_2$.
\begin{itemize}
\item Define Laplace Beltrami with respect to $g_1$ (potentially with weight)
\item Define gradients and Lipschitz constant with respect to $g_2$
\end{itemize}
The Ricci curvature bounds can now be defined either via Lipschitz decay, or via Bakry Emery calculus which gives rise to a whole universe of interesting open questions, e.g.,
\begin{itemize}
\item How does the curvature localize, i.e., how to get pointwise curvature quantities from the global bounds?
\item Under which conditions do the Lipschitz decay curvature and Bakry Emery curvature coincide?
\item Which classical results regarding Ricci curvature can be generalized to the case of having two different metrics?
\item What are natural definitions of scalar, sectional and Riemannian curvature in this context?
\item Are there interesting Ricci flows, deforming the metrics differently?
\end{itemize}

We now discuss the known relations between Ricci curvature, isoperimetry and log-Sobolev inequalities.
In smooth and discrete settings, there is hierarchy of properties with concentration inequalities and isoperimetric inequalities at its ends. Specifically,
\begin{align*}
\mbox{Isoperimetric inequalities} &\Rightarrow \mbox{Functional inequalities} \\&\Rightarrow \mbox{Transport entropy inequalites}\\& \Rightarrow  \mbox{Concentration inequalities},
\end{align*}
see e.g. \cite{milman2012properties} for the smooth case and \cite{fathi2015curvature} for the discrete case. In the seminal works of Milman \cite{milman2009role,milman2010Isoperimetric}, it is shown that in the smooth case, the hierarchy can be reversed in case of non-negative Ricci curvature. In particular, exponential concentration of measure implies a lower bound on the Cheeger constant, and Gaussian concentration implies Gaussian isoperimetry. In \cite[Conjecture~6.9 and 6.10]{erbar2018poincare}, Erbar and Fathi ask whether the same implications hold true in the discrete setting.
In the smooth setting, the log Sobolev constant can be lower bounded by the minimal Ricci curvature
\cite{cavalletti2017sharp}.
This bound is sharp and equality is attained iff the space splits out a 1-dimensional Gaussian space \cite{ohta2020equality}.

In the study of discrete Markov chains, log-Sobolev inequalities are of fundamental importance  as they provide a crucial tool to investigate mixing times and cutoff behavior \cite{diaconis1996logarithmic}.

It was conjectured by Tetali and Peres that in the discrete setting, lower Ollivier curvature bounds imply modified log-Sobolev inequalities \cite[Conjecture~3.1]{eldan2017transport}, \cite[Conjecture~4]{fathi2019quelques}, \cite[Remark~1.1]{blanca2022mixing}, \cite[Conjecture~5.25]{pedrotti2023contractive}.

There is a wide range of literature attempting to solve the conjecture: In \cite{eldan2017transport,fathi2015curvature}, it was tried to tackle the conjecture via transport inequalities. In \cite{blanca2022mixing}, log-Sobolev inequalities under positive curvature bounds have been proven for Glauber dynamics.
In \cite{johnson2015discrete}, a log Sobolev inequality was proven for birth death chains with constant birth rate, assuming a Bakry Emery curvature bound.

Recently, Pedrotti conjectured that positive Ollivier curvature implies a modified log-Sobolev inequality if additionally assuming non-negative $\ell_\infty$ Wasserstein curvature (i.e., the Ollivier curvature when replacing the $\ell_1$ by the $\ell_\infty$ Wasserstein curvature), see the discussion after \cite[Conjecture~5.25]{pedrotti2023contractive}. In \cite[Problem~P]{ollivier2010survey}, Ollivier asks whether the $\ell_\infty$ Wasserstein curvature has any interesting applications. As (unmodified) log-Sobolev inequalities seemed completely out of reach, Ollivier asked if one could get at least exponential decay of the entropy under the heat flow \cite[Problem~M]{ollivier2010survey}.

While it was completely open in case of Ollivier curvature bounds, modified log-Sobolev inequalities have been proven under entropic curvature bounds \cite{erbar2018poincare,erbar2012ricci}, and under the exponential curvature dimension condition $CDE'$ \cite{yong2017log}. However, the entropic curvature and the $CDE'$ curvature condition are highly non-linear as the logarithm is appearing. In particular, they are practically not computable for larger networks (i.e., networks with more than two vertices), although useful bounds have been given for important classes of Markov chains.

It turns out that under Ollivier or Bakry Emery curvature conditions, estimates on graphs involving the logarithm are notoriously hard to prove.  This has been impressively demonstrated in the papers on Li-Yau inequalities on graphs, which all but one have to assume a non-linear modification of the Bakry Emery condition \cite{bauer2015li,horn2014volume,munch2014li,
dier2017discrete,gong2019li,qian2017remarks,weber2023li,
lippner2016li,krass2022li,horn2019spacial}. The only exception is \cite{munch2019li} in which a Li-Yau inequality under the standard Bakry Emery condition was proven for a modified heat equation so that the appearance of the logarithm was sneakily prevented.

While log-Sobolev inequalities are hard to obtain from Ollivier curvature bounds, the weaker Poincare inequality is well known to be true. Specifically, if $K$ is a lower Ollivier curvature bound, then
\[
\lambda_1 \geq K,
\]
see \cite{ollivier2009ricci}, 
and in case of non-negative Ollivier curvature,
\[
\lambda_1 \geq \frac{\log(2)}{\diam^2},
\]
where $\lambda$ is the smallest positive eigenvalue of the graph Laplacian, see \cite{munch2019non}.
A similar estimate has been shown in \cite{lin2010ricci} in case of non-negative Bakry Emery curvature.
The leading question for this article was, whether $\lambda_1$ in the above estimates can be replaced by the (modified) log-Sobolev
constant.
It turns out that the methods of the proofs also yield surprisingly strong isoperimetric inequalities.

\subsection{Main results for positive curvature}

We give a simple example that the Peres-Tetali conjecture is wrong, i.e., we construct a sequence of Markov chains with uniformly positive Ollivier and Bakry Emery curvature, but the modified log-Sobolev constant tends to zero, see Section~\ref{sec:CounterExample}.
The Markov chain is a birth death chain on three vertices where the transition rate from an outer to the middle vertex goes to zero and the remaining transition rates stay fixed. It turns out that with appropriate choice of the remaining transition rates, the Markov chain has uniformly positive Ollivier and Bakry Emery curvature.
Moreover, based on this birth death chain, we construct a combinatorial graph with uniformly positive Ollivier curvature, but arbitrarily small modified log-Sobolev constant.

Moreover, we present a method to repair the Peres-Tetali conjecture by employing a notion of Ollivier sectional curvature based on the $\ell_\infty$ Wasserstein distance instead of its $\ell_1$ version (for the motivation of the sectional curvature, see \cite[Problem P]{ollivier2010survey}).
Specifically, assuming non-negative Ollivier sectional curvature, we prove in Theorem~\ref{thm:modLogSobPosCurv} that 
\[
\amod \geq \inf_{x\sim y} \kappa(x,y)
\]
where $\kappa$ is the Ollivier Ricci curvature and $\amod$ the modified log-Sobolev constant.
This settles an open question of Pedrotti (see the discussion after \cite[Conjecture~5.25]{pedrotti2023contractive}).

\subsection{Main results for non-negative curvature}
The deepest and most innovative results of this article are concerning non-negative Ollivier Ricci curvature and its consequences for the isoperimetric profile.
Specifically, we present a Laplacian separation principle
for the eikonal equation which, at first glance, looks just like a party trick, see Section~\ref{sec:LaplaceSeparation}. 
However, this method seems entirely new (even in the case of Riemannian manifolds) and it has surprisingly strong consequences for the isoperimetric profile.
More precisely, in Theorem~\ref{thm:IsoperimetryConcentration}, we prove that for every $\eps \leq \frac 1 8$ and every vertex subset $W$ with $m(W) \leq \frac 1 2$,
\[
|\partial W| \geq P_0\frac{(m(W) \wedge \eps) \log (1/\eps)}{6\diamess^{(\eps)}}.
\]
Here, $m$ is the reversible probability measure, $P_0$ is the minimal transition rate, $|\partial W|$ is the size of the edge boundary, and $\diamess$ is the observable diameter, i.e.,
\[
\diamess^{(\eps)} = \sup_{\substack{m(A) \geq \eps\\m(\cl(B)) \geq \eps}} d(A,B)
\]
where $\cl(B)$ is the closure containing both $B$ and its outer vertex boundary.
The observable diameter as a function of $\eps$ is, up to constant factors, the inverse of the concentration profile, see \cite{ozawa2015estimate}. Therefore, the above estimate gives an intimate relation between the isoperimetric profile (depending on $|\partial W|$ and $m(W)$), and the concentration profile (depending on $\eps$ and $\diamess^{(\eps)}$).
Particularly, every single point of the concentration profile (with the only restriction that $\eps\leq 1/8$) yields an estimate for every point of the isoperimteric profile.
This is indeed a conceptual improvement compared to the seminal work of Milman \cite{milman2010Isoperimetric}, where the full concentration profile, and particularly its tail is needed to obtain isoperimetric estimates.

We now discuss the consequences of the above isoperimetry-concentration estimate.
Plugging in $\eps =m(W)/4$, we obtain asymptotically sharp estimates for the log-Cheeger constant,
\[
h_{\log} := \inf_{m(W)\leq 1/2} \frac{|\partial W|}{ - m(W) \log(m(W))} \geq \frac{P_0}{24\diam(G)},
\] 
see Corollary~\ref{cor:IsoperimetryConcentrationLargeEps}.
This seems to be the first result of this kind for discrete Ricci curvature bounds. The result is asymptotically sharp for uniformly biased and unbiased birth death chains.

When plugging in $\eps=1/8$, we obtain the following estimate for the classical Cheeger constant,
\[
h := \inf_{m(W) \leq 1/2} \frac{|\partial W|}{m(W)}  \geq \frac
{P_0}{12 \diamess^{(1/8)}},
\]
see Corollary~\ref{cor:hlargerdiamess}.
This means that the Cheeger constant can be lower bounded in terms of a single point of the concentration profile.
This indeed exceeds Erbar and Fathi's conjecture \cite[Conjecture~6.9]{erbar2018poincare} asking whether the Cheeger constant and spectral gap can be bounded assuming exponential concentration, see Section~\ref{sec:ErbarFathiConjecture}.
Indeed, the reverse inequality is also true, namely
\[
h \leq \frac
{57\D}{12 \diamess^{(1/8)}}
\]
where $\D=\sum_{y\neq x} P(x,y)$ is the maximal vertex degree, see Theorem~\ref{thm:hsmallerdiamess}.

We also obtain a lower bound of $|\partial W|/m(W)$ in terms of the diameter of $W$. More precisely, in Theorem~\ref{thm:internalDiameter}, we prove that
\[
\frac{|\partial W|}{m(W)(1-m(W))} \geq \frac{P_0}{\diam(\cl(W))}. 
\]

Another application of our general isoperimetric concentration inequality is that we can prove Gaussian isoperimetry assuming Gaussian concentration.
Specifically, assume that for all $A$ with $m(A) \geq \frac 1 2$,
\[
m(A_r) \leq \exp(-\rho r^2)
\]
where $A_r = \{x:d(A,x) \leq r\}$.
Then,
\[
h_{\sqrt {\log}} := \inf_{m(W)\leq \frac 1 2} \frac{|\partial W|}{m(W)\sqrt{\log(1/m(W))}}  \geq \frac{P_0}{48}\sqrt{\rho}.
\]

The $\sqrt{\log}$ Cheeger constant (or Gaussian isoperimetric constant) is tightly related to the log-Sobolev constant via log-Cheeger Buser inequalities, see \cite[Theorem~4.4]{klartag2015discrete} and \cite[Remark 5]{houdr2001mixed}, i.e., $Ch_{\sqrt{\log}}^2/P_0 \geq \alpha \geq c h_{\sqrt{\log}}^2$, where the former estimate requires non-negative Ricci curvature. By this, we provide an answer to \cite[Conjecture~6.10]{erbar2018poincare} where Erbar and Fathi ask if Gaussian concentration implies a log Sobolev inequality.

One drawback when applying log-Cheeger estimates for obtaining log-Sobolev inequalities is that one loses a factor $P_0^2$. This is indeed avoidable when estimating the log-Sobolev constant purely in terms of the diameter.
To address this issue, we give a new general characterization of the log-Sobolev constant $\alpha$ in terms of Dirichlet eigenvalues. Specifically, in Corollary~\ref{cor:asalpha}, we show
\[
\alpha \simeq \as := \inf_{W} \theta(\lambda_W,\lambda_{W^c})
\]
where $\lambda_W$ is the smallest Dirichlet eigenvalue of the subset $W$, and $\theta$ is the logarithmic mean.
Here, $\simeq$ means coincidence up to global multiplicative constants.
Using the spectral characterization of $\alpha$, we prove in Theorem~\ref{thm:asDiam} that
\[
\alpha \simeq \as \geq \frac {P_0} {16\diam(G)^2}.
\]

Considering that under discrete curvature assumptions, only modified log-Sobolev inequalities have been established so far \cite{erbar2012ricci,erbar2018poincare,weber2021entropy,
yong2017log,johnson2015discrete,caputo2009convex,
Erbar2019EntropicCA}, it is quite remarkable that our approach yields original log-Sobolev  inequalities which are significantly stronger and have the advantage that they are closely related to the time to stationarity of random walks, and that they allow for geometric characterizations via isocapacitary and spectral profiles, as discussed in Section~\ref{sec:GeneralLogSob}.
Additionally, this addresses \cite[Problem~M]{ollivier2010survey} asking for exponential entropy decay under the heat flow which is a well known consequence of log-Sobolev inequalities.

Moreover it is surprising that we obtain the log-Sobolev inequality and isoperimetric inequalities under Ollivier curvature bounds which are easy to determine by a linear program, and not under modified Bakry Emery curvature bounds such as $CDE'$ or entropic curvature bounds which come with logarithmic terms facilitating the proof of log-Sobolev inequalities.

\section{Setup and notation}
We say $G=(V,P,d)$ is a metric (continuous time) Markov chain if
\begin{itemize}
\item The vertex set $V$ is a finite set,
\item The Markov kernel $P:V^2 \to [0,\infty)$ has symmetric support,
\item The distance function $d:V^2 \to [0,\infty]$ satisfies 
\[
d(x,y)=\inf \left\{\sum_{k=1}^n d(x_k,x_{k-1}):x=x_0 \sim \ldots \sim x_n = y \right\}
\]

where we write $x\sim y$ iff $P(x,y)>0$.
\end{itemize}
In general, we do not require that the $P(x,\cdot)$ sums to one. However, when we refer to lazy Markov chains, we require $\sum_{y}P(x,y)=1$ and $P(x,x) \geq \frac 1 2$ for all $x \in V$.
The last condition above means that $d$ is a discrete path distance, i.e., the distance between between any two vertices is realized by the length of a path with respect to the Markov chain.

We always assume that the graph is connected, i.e., $d(x,y)$ is finite for all $x,y$.
Then, the Markov chain is irreducible, i.e., there exists a unique probability measure $m$ on $V$, called invariant measure, such that for all $x \in V$,
\[
\sum_y m(x)P(x,y) = \sum_y m(y)P(y,x).
\]
If for all $x,y \in V$, we have $m(x)P(x,y) = m(y)P(y,x)$, then the Markov chain is called reversible. We sometimes write $w(x,y)=m(x)P(x,y)$, referring to the symmetric edge weight of a weighted graph.
We will always assume that the Markov chain is reversible.
The measure $m$ induces a scalar product on $\R^V$ via
\[
\langle f,g \rangle := \sum_x f(x)g(x)m(x).
\]

We denote the 1-Lipschitz functions by
\[
\Lip(1) := \{f \in \R^V: |f(y)-f(x)| \leq d(x,y)\}.
\]
The Laplacian is given by $\Delta: \R^V \to \R^V$,
\[
\Delta f(x) = \sum_y P(x,y)(f(y)-f(x)).
\]
It is well known that $\Delta$ is self-adjoint if and only if the Markov chain is reversible.
For reversible Markov chains, the corresponding  Dirichlet form is given by
\[
\mathcal E(f,g) = - \langle \Delta f,g \rangle.
\]
For simplicity, we write $\mathcal E(f):=\mathcal E(f,f)$.
For introducing the Ollivier curvature, we follow \cite{munch2017ollivier}.
For $x\sim y$ the Ollivier curvature $\kappa(x,y)$ is given by
\[
\kappa(x,y) := \inf_{\substack{f \in \Lip(1) \\ f(y)-f(x)=d(x,y)}} \frac{\Delta f(x) -\Delta f(y)}{d(x,y)}.
\]
If $P$ is a lazy Markov kernel, i.e., if $P(x,x) \geq \frac 1 2$, and if $\sum_y P(x,y)=1$ for all $x \in V$, the Ollivier curvature coincides with the expression introduced by Ollivier, namely
\[
\kappa(x,y) = 1 - \frac {W(P(x,\cdot),P(y,\cdot))}{d(x,y)}
\]
where $W$ is the 1-Wasserstein distance, see \cite{bourne2018ollivier,loisel2014ricci}.
We remark that the Ollivier curvature of an edge coincides with the maximal Forman curvature of this edge where the maximum is taken over all cell complexes having the graph as 1-Skeleton \cite{jost2021characterizations}.

\section{General characterization of log-Sobolev constants}
\label{sec:GeneralLogSob}

We first give the definitions of the log-Sobolev constant and its modified version.
The Log-Sobolev constant $\alpha$ is defined as
\[
\alpha := \inf_{\|f\|_2=1}  \frac{ \mathcal E(f,f)}{\Ent(f^2)}
\]
where for positive $f$ with  $\|f\|_1=1$,
\[
\Ent(f) := \langle f,\log f \rangle.
\]

It is easy to check that for the complete graph on $n$ vertices (with normalized graph Laplacian), $\alpha$ goes to zero as $n$ increases.
This is for some applications not desirable which was one motivation to introduce a modified log Sobolev constant
\[
\amod := \inf_{\|f\|_1=1} \frac{\mathcal E(f,\log f)}{\Ent(f)}.
\]
It is well known that
\[
4 \alpha \leq  \amod \leq 2\lambda
\]
where $\lambda$ is the smallest positive eigenvalue of $-\Delta$, see e.g. \cite{bobkov2006modified}.
By a variational argument, if $2\amod<\lambda$, then the minimizer for $\amod$ satisfies
\[
\frac {\Delta f} f + \Delta (\log f) = -\amod \log f,
\]
see \cite[Theorem~6.5]{bobkov2006modified}.
If $\amod = 2\lambda$, then the expression for $\amod$ is minimized by $1+\eps f$ for $\eps \to 0$ where $f$ is an eigenfunction to $\lambda$.

The log Sobolev constant is closely related to the  time to stationarity $\tau$ of a random walk. By a famous result of Diaconis and Saloff-Coste \cite{diaconis1996logarithmic}, we have
\[
\frac 1 {2\alpha} \leq \tau \leq \frac {4 + \log \log (1/\pi_*)}{4\alpha}
\] 
where 
$\tau := \inf \{t: \sup_x \|e^{\Delta t} \delta_x - 1 \|_2 \leq \frac 1 e\}.$
Moreover, the log Sobolev constant $\alpha$ allows for various geometric and spectral characterizations as discussed below.

In contrast, the modified log-Sobolev constant $\amod$ can be seen as more analytic in nature as it precisely measures the decay of entropy under the heat equation \cite{bobkov2006modified}.

In the next subsection, we discuss a capacitary characterization of $\alpha$ from \cite{schlichting2019poincare}. Afterwards, we give a characterization of $\alpha$ in terms of Dirichlet eigenvalues. This seems to be new.

\subsection{Capacity and Log-Sobolev constant}

The log-Sobolev constant is closely related to the iso-capacitary profile which can be seen as a refined version of the isoperimetric profile. 
Here, 'refined' means that that isocapicatary constants give upper and lower bounds to the spectral gap and log-Sobolev constant up to a global constant, while isoperimetric constants yield similar estimates via Cheeger-Buser inequality only up to a constant depending on the vertex degree.
We now give the details.
Let $A,B \subset V$ be disjoint subsets.
We define
\[
\capa(A,B) := \min \{\mathcal E(f): f|_A=0, f|_B=1\}.
\]
We assume $m(V)=1$.
Let
\[
\alpha_{\capa} := \inf_{\substack{A,B \subset V \\m(B) \geq 1/2}}\frac{\capa(A,B)}{-m(A)\log \left(1 + \frac {e^2} {m(A)}\right)}
\]

The following theorem is given in
\cite[Corollary~2.11]{schlichting2019poincare}.
\begin{theorem}\label{thm:acapalpha}
There exist universal constants $c,C$ such that
\[
c\alpha_{\capa} \leq \alpha \leq C \alpha_{\capa}.
\]
\end{theorem}
For the proof, we refer to \cite{schlichting2019poincare}.
We first bring $\alpha_{\capa}$ into a more convenient form by introducing
\[
\alpha_{\capa}^{\theta} := \inf_{A,B \subset V} \theta \left(\frac{\capa(A,B)}{m(A)},\frac{\capa(A,B)}{m(B)} \right)
\]
where $\theta$ is the logarithmic mean, i.e., $\theta(s,t)=(s-t)/\log(s/t)$ for $s\neq t$ and $\theta(s,s)=s$.
We now compare $\alpha_{\capa}$ with $\alpha_{\capa}^\theta$.
We remark that most of the following proof boils down to single variable calculus, but for convenience of the reader, we give a complete proof.
\begin{lemma}\label{lem:acaptheta}The isocapacitary constants $\alpha_{\capa}$ and $\alpha_{\capa}^\theta$ satisfy
\[
\frac 1 2 \alpha_{\capa} \leq \alpha_{\capa}^\theta \leq 3 \alpha_{\capa}.
\]
\end{lemma}
\begin{proof}
We calculate
\[
\theta \left(\frac{\capa(A,B)}{m(A)},\frac{\capa(A,B)}{m(B)} \right)= \frac{\capa(A,B)}{m(A)} \cdot \frac{1-\frac {m(A)}{m(B)}}{\log(m(B)/m(A))}.
\]
We first show the latter inequality.
Assume $A,B$ minimize $\alpha_{\capa}$. Then, $m(A) \leq \frac 1 2 \leq m(B) \leq 1$ and thus,
\[
\frac{1-\frac {m(A)}{m(B)}}{\log(m(B)/m(A))} \leq \frac{1- 2 m(A) }{-\log(2m(A))} \leq 3 \cdot \frac{1}{-\log \left(1 + \frac {e^2} {m(A)}\right)}
\]
and therefore,
\[
\alpha_{\capa}^{\theta} \leq \theta \left(\frac{\capa(A,B)}{m(A)},\frac{\capa(A,B)}{m(B)} \right) \leq 3 \cdot \frac{\capa(A,B)}{-m(A)\log \left(1 + \frac {e^2} {m(A)}\right)} = 3 \alpha_{\capa}. 
\]
We now show the former inequality.
We assume $A,B$ minimize $\alpha_{\capa}^\theta$.
Let $h$ be such that $h_A=0$ and $h_B=1$ and $\Delta h=0$ on $(A\cup B)^c$. Let $B' := \supp (h-\frac 1 2)_+$. Without obstruction, we can assume $m(B') \geq \frac 1 2$, as otherwise, we could exchange $A$ and $B$.
 Then,
 \[
\capa(A,B') \leq 4 \mathcal E((h-\frac 1 2)_+) \leq 2\capa(A,B).
 \]
Moreover,
\[
\frac{1}{-\log \left(1 + \frac {e^2} {m(A)}\right)} \leq 
\frac{1- A }{-\log(A)}  \leq  \frac{1-\frac {m(A)} {m(B)}}{\log(m(B)/m(A))}
\]
and thus,
\begin{align*}
\alpha_{\capa} \leq \frac{\capa(A,B')}{m(A)}\frac{1}{-\log \left(1 + \frac {e^2} {m(A)}\right)}  &\leq \frac{2\capa(A,B)}{m(A)}\frac{1-\frac {m(A)} {m(B)}}{\log(m(B)/m(A))} \\&= 2\theta \left(\frac{\capa(A,B)}{m(A)},\frac{\capa(A,B)}{m(B)} \right) =2 \alpha_{\capa}^\theta.
\end{align*}
This implies the first estimate and finishes the proof.
\end{proof}

\subsection{Capacity and Dirichlet eigenvalues}

Estimates of the log-Sobolev constant via the spectral profile are given in \cite{goel2006mixing,hermon2018characterization}.
In their work, they compare the Dirichlet eigenvalue with the volume.
We compare the Dirichlet eigenvalue with the Dirichlet eigenvalue of the complement.
Let $X \subset V$.
We define
\[
\lambda_X := \inf \{\mathcal E(f,f): \|f\|_2 = 1, f|_{X^c} = 0  \},
\]
and we define the log-spectral constant as
\[
\as := \inf_{X \subset V} \theta(\lambda_X,\lambda_{X^c}).
\]
where $\theta$ is the logarithmic mean. 

\begin{theorem}\label{thm:asacap}
The  log-spectral constant $\as$ satisfies
\[
\frac 1 4\alpha_{\capa}^\theta \leq \as \leq 2\alpha_{\capa}^\theta. 
\]
\end{theorem}

\begin{proof}
We first prove the latter inequality.
Let $A,B \subset V$ minimizing $\alpha_{\capa}^\theta$.
Let $h$ be the minimizer of the capacity, i.e., $\Delta h= 0$ on $(A\cup B)^c$ and $h=0$ on $A$ and $h=1$ on $B$. 
Let $X:= \{h \leq \frac 1 2\}$ and $Y=X^c$.
Let $h_X := (h - \frac 1 2)_-$ and $h_Y := (h- \frac 1 2)_+$
Then,
\[
\frac 1 2 \mathcal E(h) \geq \mathcal E(h_X) \geq \lambda_X \|h\|_2^2 \geq \frac {\lambda_X} 4 m(A)
\]
and hence,
\[
\lambda_X \leq  \frac{2 \capa(A,B)}{m(A)}
\]
and similarly,
\[
\lambda_Y \leq  \frac{2 \capa(A,B)}{m(B)}.
\]
Thus,
\begin{align*}
\as \leq \theta(\lambda_X,\lambda_{Y}) \leq 2\theta \left(\frac{\capa(A,B)}{m(A)},\frac{\capa(A,B)}{m(B)} \right) =2 \alpha_{\capa}^\theta.
\end{align*}
We finally prove the former estimate.
Let $X$ be a minimizer of $\as$ and let $y=X^c$.
By \cite[Corollary~2.10]{schlichting2019poincare}, there exists $A \subset X$ with
\[
\frac{\capa(A,Y)}{m(A)} \leq 4 \lambda_X
\]
Similarly, there exists $B \subset Y$ such that
\[
\frac{\capa(B,X)}{m(B)} \leq 4 \lambda_Y
\]
Using $\capa(A,B) \leq \capa(A,Y)$ and $\capa(A,B) \leq \capa(B,X)$, we obtain
\[
\alpha_{\capa}^\theta \leq \theta \left(\frac{\capa(A,B)}{m(A)},\frac{\capa(A,B)}{m(B)} \right) \leq 4\theta(\lambda_X,\lambda_Y) = 4\as.
\]
This proves the former inequality and finishes the proof.
\end{proof}

Combining Thereom~\ref{thm:acapalpha}, Lemma~\ref{lem:acaptheta} and Theorem~\ref{thm:asacap}, we obtain the following corollary.
\begin{corollary}\label{cor:asalpha}
There exist global constants $c,C>0$ such that
\[
c\alpha \leq \as \leq C \alpha.
\]
\end{corollary}

\subsection{Dirichlet eigenvalue and positive super-solutions}

We recall a general lower bound on the Dirichlet eigenvalue in terms of positive super solutions.
This kind of result is known as Allegretto Piepenbrink type result, \cite{haeseler2011generalized}. We first set the stage. 
Let $W \subsetneq V$, let
\[
\Delta_W f := 1_W \Delta (1_W \cdot f),
\]
and let $\lambda_W$ be the smallest positive eigenvalue of $-\Delta_W$.
The following Lemma can be found in \cite[Theorem~3.1]{haeseler2011generalized} in a way more general framework. For convenience, we give a simple proof for our framework.
\begin{lemma}\label{lem:DirichletSupersolution}
Let $\lambda>0$. Suppose there is a non-negative function $f \in \R^V$ with $f|_W \neq 0$  and
\[
\Delta f \leq -\lambda f \mbox{ on } W.
\]
Then, $\lambda_W \geq \lambda$.
\end{lemma}

\begin{proof}
Without obstruction, we assume $f=0$ on $V\setminus W$.
We notice that $f>0$ on $W$ by the maximum principle.
Let $u_t := e^{-\lambda t} f$.
Then,
$\partial_t u_t \geq \Delta_W u_t$.
Let $g$ be an eigenfunction with respect to $\lambda_W$ such that $g \leq f$. Let $v_t := e^{-\lambda_W t} g$.
Then, $\partial_t v_t  = \Delta_W v_t$.
Hence,
\[
\partial_t (u_t - v_t) \geq \Delta(u_t - v_t)
\]
and $u_0 - v_0 \geq 0$. This implies
$ u_t - v_t \geq 0 $
by the maximum principle and thus, $\lambda_W \geq \lambda$, finishing the proof.
\end{proof}

\section{Positive curvature}\label{sec:PosCurv}

In this section, we investigate the relation between positive Ollivier curvature and modified log-Sobolev inequalities.
Specifically, we prove that for Markov chains,
\[
\Ric \geq K  \mbox{ \&  } \sec \geq 0 \quad \Longrightarrow \quad \amod \geq K
\]
where $\amod$ is the modified log-Sobolev constant, $\Ric$ is the Ollivier Ricci curvature, and $\sec$ is the Ollivier sectional curvature based on the $\ell_\infty$ Wasserstein distance discussed in the next subsection.
This answers an open question by Pedrotti (see the discussion after \cite[Conjecture~5.25]{pedrotti2023contractive}).
 The proof is based on a new characterization of non-negative sectional curvature via various non-linear gradient estimates, see Corollary~\ref{cor:secCharGradientEstimates}.

In Subsection~\ref{sec:CounterExample}, we prove that non-negative sectional curvature is necessary for obtaining the modified log-Sobolev inequality. More precisely, we give an example of a Markov chain with
\begin{align*}
\Ric \geq 1 \quad \mbox{ and } \quad \amod \leq \eps
\end{align*}
for arbitrarily small $\eps>0$.
This example serves as counterexample for a conjecture by Peres and Tetali.

\subsection{Ollivier sectional curvature}
In Yann Ollivier's survey of Ricci curvature for metric spaces and Markov chains \cite[Problem~P]{ollivier2010survey}, he proposes a notion of discrete sectional curvature by replacing the $\ell_1$ by the $\ell_\infty$ Wasserstein metric in the formula for the Ollivier curvature. Specifically for having non-negative sectional curvature, there must exist a coupling between $P(x,\cdot)$ and $P(y,\cdot)$ moving all points by a distance of at most $d(x,y)$. We now give the precise definition.

\begin{definition}
We say a lazy Markov  chain $(V,P)$ has non-negative Ollivier sectional curvature at edge $x\sim y$ if
there exists a transport plan $\pi:V\times V \to [0,\infty)$ transporting the measure $P(x,\cdot)$ to $P(y,\cdot)$ such that
\[
d(x',y') \leq 1
\] 
whenever $\pi(x',y')>0$.
\end{definition}

Equivalently, one can introduce a sectional curvature lower bound using the $\ell_\infty$ Wasserstein distance, i.e.,
\[
\kappa_\infty(x,y) := 1 - \frac{W_\infty(P(x,\cdot),P(y,\cdot))}{d(x,y)}
\]
where
\[
W_\infty(\mu,\nu) = \inf_{\pi} \sup_{(x',y')\in \supp(\pi)}d(x',y')
\]
where the infimum is taken over all transport plans $\pi$.
It is easy to check that this definition precisely coincides with non-negative curvature in
\cite[Definition~5.3]{pedrotti2023contractive}
where the curvatures with respect to different $p$-Wasserstein metrics have been compared.

Moreover, the sectional curvature introduced above is closely related to various non-linear Ollivier type curvature notions introduced in \cite{kempton2019large}.
Particularly, we show that non-negative sectional curvature is equivalent to all statements in
\cite[Theorem~1.12(iii)]{kempton2019large}.
\begin{corollary}\label{cor:secCharGradientEstimates}
Let $(V,P)$ be a lazy reversible Markov chain equipped with the combinatorial distance.
The following statements are equivalent:
\begin{enumerate}[(i)]
\item $G$ has non-negative sectional curvature 
\item $\|\nabla \log P_t f\|_\infty \leq \|\nabla \log f\|_\infty$ for all positive $f\in \R^V$.
\item $|\nabla P_t f| \leq P_t |\nabla f|$ for all $f \in \R^V$
\item $|\nabla \sqrt{P_t f}| \leq P_t |\nabla \sqrt f|$ for all positive $f\in \R^V$,
\end{enumerate}
where $|\nabla f|(x):= \max_{y\sim x} |f(y)-f(x)|$ is the pointwise gradient and $P_t = e^{\Delta t}$ is the heat semigroup.
\end{corollary}

The corollary is an easy consequence of the following theorem and \cite[Theorem~1.12(iii)]{kempton2019large} combined with \cite[Theorems~1.6,1.7,1.8]{kempton2019large}.

\begin{theorem}\label{thm:secCurvatureExp}
Let $(V,P)$ be a lazy reversible Markov chain equipped with the combinatorial distance.
Let $x\sim y \in V$. 
The following statements are equivalent.
\begin{enumerate}[(i)]
\item The edge $(x,y)$ has non-negative Ollivier sectional curvature, 
\item For all functions $f \in \Lip(1)$ with $f(y)-f(x)=1$, and all $\lambda \geq 0$,
\[
\frac{\Delta e^{\lambda f}(x)}{e^{\lambda f}(x)} \geq \frac{\Delta e^{\lambda f}(y)}{e^{\lambda f}(y)} \qquad \mbox{ and } \qquad \frac{\Delta e^{-\lambda f}(x)}{e^{-\lambda f}(x)} \leq \frac{\Delta e^{-\lambda f}(y)}{e^{-\lambda f}(y)}.
\]
\end{enumerate}
\end{theorem}

\begin{proof}
We first prove $(i) \Rightarrow (ii)$.
Let $\pi$ be a transport plan with transport distance at most one.
We notice that whenever $\pi(x',y')>0$, we have
\[
f(x')-f(x) \geq f(y')-f(y)
\]
as $f \in \Lip(1)$ and as $f(y)-f(x)=1$.
We calculate
\begin{align*}
&\frac{\Delta e^{\lambda f}(x)}{e^{\lambda f}(x)} - \frac{\Delta e^{\lambda f}(y)}{e^{\lambda f}(y)}
\\=&\sum_{x',y'} \pi(x',y') \left(\exp(\lambda(f(x')-f(x))) - \exp(\lambda(f(y')-f(y))) \right)
\geq 0
\end{align*}
as all terms are non-negative. The other inequality follows similarly.

We now prove $(ii) \Rightarrow (i)$.
We consider the optimal transport problem 
\[
\min_\pi \sum_{x',y'} c(x',y')\pi(x',y')
\]
over all transport plans $\pi$ from $P(x,\cdot)$ to $P(y,\cdot)$, where $c(x',y')=0$ if $d(x',y')\leq 1$ and $c(x',y')=\infty$ otherwise.
The dual problem is
\[
\max_{f,g} Pg(y) - Pf(x)
\]
where the maximum is taken over all $f,g \in \R^V$ such that $g(y') \leq f(x')$ whenever $d(x',y') \leq 1$.
We aim to show that the maximum is zero. Suppose not.
We assume without loss of generality that $f(x)=0$.
As $P(z,z) \geq \frac 1 2$, we can also assume $g(y)=0$, as otherwise, we could replace $f$ by its positive part and enforce $g(y)$ to be zero.
Also, without loss of generality, we assume
\[
Pg_+(y) > Pf_+(x)
\]
as otherwise, we could interchange $x$ and $y$.
By a level set argument, there exists $r>0$ such that 
\[
P1_{\{g>r\}}(y) > P1_{\{f>r\}}(x).
\]
As $g \leq f(y)$ on $B_1(y)$, we see $y \in \{f>r\}$.

We now construct a 1-Lipschitz function $h$ as 
\[
h:=1_{\{g>r\}} - 1_{\{f \leq r\}}.
\]
Indeed $h \in \Lip(1)$ as $f(x')<g(y')$ implies $d(x',y')\geq 2$.
We calculate
\[
\frac{\Delta e^{\lambda h}}{e^{\lambda h}}(x) = (e^\lambda - 1) P1_{\{f>r\}} (x)
\]
and
\[
\frac{\Delta e^{\lambda h}}{e^{\lambda h}}(y) = (e^\lambda - 1)  P1_{\{g>r\}}(y) + (e^{-\lambda} - 1) P1_{\{f \leq r\}}(y).
\]
As $P1_{\{g>r\}}(y) > P1_{\{f>r\}}(x)$, we see that for $\lambda$ large, we have
\[
\frac{\Delta e^{\lambda h}}{e^{\lambda h}}(y) > \frac{\Delta e^{\lambda h}}{e^{\lambda h}}(x)
\] 
contradicting $(ii)$. This proves $(ii) \Rightarrow (i)$ by contradiction.
\end{proof}

\subsection{Ollivier curvature and modified log-Sobolev inequality}

We show that the modified log-Sobolev constant can be lower bounded by the Ollivier Ricci curvature, assuming non-negative sectional curvature. 
This answers an open question by Pedrotti affirmatively, see  \cite[Conjecture~5.25]{pedrotti2023contractive} and the discussion thereafter.
Indeed, non-negative sectional curvature is crucial as will be shown in an example on a three vertex birth death chain.

\begin{theorem}\label{thm:modLogSobPosCurv}
Let $(V,P)$ be a Markov chain equipped with the combinatorial distance. Assume that $(V,P)$ has non-negative Ollivier sectional curvature. Then,
\[
\amod \geq \inf _{x\sim y} \kappa(x,y).
\]
\end{theorem}

\begin{proof}
Let $K:=\inf _{x\sim y} \kappa(x,y)$.
As $\lambda \geq K$, see e.g. \cite{ollivier2009ricci}, we can assume $\amod < 2\lambda$ without obstruction.
Hence, there exists a non-constant function $f$ satisfying $\|f\|_1=1$ and
\[
\frac {\Delta f} f + \Delta (\log f) = -\alpha \log f,
\]
see \cite[Theorem~6.5]{bobkov2006modified}.
Let $g= \log f$.
Let $x\sim y$ such that
\[
g(y) - g(x) = \|\nabla g \|_\infty.
\]
By non-negative Ollivier sectional curvature and Theorem~\ref{thm:secCurvatureExp}, we have
\[
\frac{\Delta e^g}{e^g}(y) \leq \frac{\Delta e^g}{e^g}(x).
\]
Moreover,
\[
\Delta g(y) - \Delta g(x) \leq - \kappa(x,y)\|\nabla g\|_\infty \leq -K\|\nabla g\|_\infty.
\]
Hence,
\begin{align*}
-\alpha \|\nabla g\|_\infty = 
\left(\frac {\Delta f} f + \Delta (\log f) \right)(y) - \left(\frac {\Delta f} f + \Delta (\log f) \right)(x)
\leq -K \|\nabla g\|_\infty
\end{align*}
implying $\alpha \geq K$ and finishing the proof.
\end{proof}

We remark that in case of birth death chains, Theorem~\ref{thm:modLogSobPosCurv} precisely recovers
\cite[Theorem~3.1]{caputo2009convex} as monotonicity of the jump rates is equivalent to non-negative sectional curvature.

\subsection{A counterexample}\label{sec:CounterExample}

In this section, we disprove the Tetali-Peres conjecture, namely that the log-Sobolev constant can be lower bounded by constant times Ollivier curvature.
In order to construct a counterexample, we must ensure that the graph has negative sectional curvature but positive Ricci curvature.
For $\eps>0$, we define a metric Markov chain $G_\eps=(\{1,2,3\},P,d)$ with combinatorial distance $d$
via
\begin{align*}
w(1,2)&=10,&
w(2,3)&=1,&
w(1,3)&=0,&\\
m(1)&=1/\eps,&
m(2)&=1,&
m(3)&=1/20.
\end{align*}
and $P(x,y) := w(x,y)/m(x)$.
It is easy to check that the Ollivier curvature is at least $1$ on the two edges.
Moreover, $G_\eps$ satisfies the Bakry Emery curvature condition $CD(1,0)$ as can be computed via \cite[Proposition~2.1]{hua2023ricci}.

We now estimate the modified log-Sobolev constant.
We choose
$f(1)=\eps$ and $f(2)=1$ and $f(3)=-\log(\eps)$.
Then, by a straight forward calculation, there exist constants $c,C$ such that for all sufficiently small $\eps>0$,
\begin{align*}
\Ent (f) \geq c\left(\log \frac 1 \eps\right)^2
\end{align*}
and
\[
\mathcal{E}(f,\log f) \leq C \log\frac 1 \eps \log \log \frac 1 \eps.
\]
Hence,
$\frac{\mathcal E(f,\log f)}{\Ent (f)}$ goes to zero as $\eps$ goes to zero.
This proves that there is no constant $C$ such that the modified log-Sobolev constant is at least $CK$ where $K$ is a lower bound on the Ollivier curvature.

\begin{remark}
This example can also be extended to combinatorial graphs, i.e., graphs for which $P(x,y)=P_0$ for all $x\sim y$.
This is done by taking a Cartesian product with the complete graph $K_{2n}$ with itself, adding another complete graph $K_n$, and putting edges between all vertices $(1,i) \in K_{2n}^2$ and $j \in K_n$.
The subgraph $\{2\ldots 2n\} \times K_{2n}$ replaces the single vertex 1 from the three vertex chain above, the subgraph $\{1\} \times K_{2n}$ replaces vertex 2, and the subgraph $K_n$ replaces vertex 3. It is a bit tedious but straight forward to check that the graph has uniformly positive Ollivier curvature (positive lower bound independent of $n$). On the other hand, the log-Sobolev constant tends to zero as $n \to \infty$ by exactly the same argument as in the three vertex chain. 
\end{remark}

\section{Non-negative Curvature and Log Sobolev inequality}
In this section, we introduce the Laplacian separation principle and prove an (unmodified) log-Sobolev inequality in terms of the diameter for Markov chains with non-negative Ollivier Ricci curvature.
We now give an intuition about the key arguments:
Assume $\Delta f \leq -C$ on a subset $W\subset V$.
Then,
\begin{itemize}
\item $\lambda_W \geq \frac{C} {2\|f\|_\infty}$,
\item $|\partial W| \geq C m(W)$.
\end{itemize}
In order to construct suitable functions $f$, we use the Laplacian separation principle introduced below to find a solution to the eikonal equation $|\nabla g| = 1$ and
\[
\Delta g|_W \leq C \leq \Delta g|_{W^c}
\]
for some unknown constant $C$. 
We then set $f = \phi \circ g$ for a suitable concave increasing function $\phi:\R \to \R$.
By a version of a discrete chain rule (see Subsection~\ref{sec:ChainRule}), the bound on $\Delta g$ can be improved exploiting concavity of $\phi$ and $|\nabla g|=1$
, i.e., $\Delta f|_W \leq C' < C$.
We now give the details.

\subsection{Eikonal equation and Laplacian separation principle}\label{sec:LaplaceSeparation}

Here, we present a mean curvature inspired Laplacian separation principle based on \cite{hua2021every}.
The key motivation comes from the fact that isoperimetric subsets of a Riemannian manifold, i.e., sets minimizing the surface area given the volume, have a boundary with constant mean curvature, up to singularities, see \cite[Appendix A.1]{milman2009role} and references therein.

In order to understand mean curvature in a discrete setting, it seems hopeless to just consider subsets of the graph, as there are only finitely many such subsets. Therefore, one would not expect to find any constant mean curvature subsets in a weighted graph. Instead, the intuition about discrete mean curvature is based on the level set approach pioneered in \cite{evans1991motion} for Euclidean spaces.

Indeed, in the smooth case, the level set mean curvature is related to the eikonal equation $|\nabla f|=1$. That is, if $|\nabla f|=1$, then the mean curvature $H(x)$ at point $x$ of the level set $\{y:f(y)=f(x)\}$ satisfies 
\[
H(x)= \nabla \cdot \left(\frac{\nabla f}{|\nabla f|}\right)(x)= \Delta f(x).
\]

Hence, in order to mimic a constant mean curvature hypersurface in a discrete space, we aim to find a solution to the discrete eikonal equation $|\nabla f|=1$ such that $\Delta f = const.$ on a given level set.
While this seems not possible due to compatibility problems with the non-reversible case, we can still treat a relaxed version of this problem.
That is that $\Delta f=const.$ on a given vertex cut set (without the restriction that the vertex cut set is a level set), and that $|\nabla f|=1$ is only required outside of the vertex cut set.

More precisely, we assume, we have a partition of the vertex set $V=X \dot \cup K \dot \cup Y$ such that there are no edges between $X$ and $Y$.
Following \cite{hua2021every}, we want to find a function with constant gradient on $X\cup Y$, minimal on $X$ and maximal on $Y$, such that moreover, the Laplacian of $f$ is constant on $K$.
By non-negative Ollivier curvature, it will follow that the cut set $K$ separates the Laplacian $\Delta f$, i.e., $\Delta f|_X \geq const. \geq \Delta f|_Y$, meaning that $f$ sovles the Laplacian separation problem discussed above.

While in \cite[Lemma~3.3]{hua2021every}, it was assumed that $K$ is connected, we will drop this condition by refining the proof idea.
Moreover, we drop the condition from \cite{hua2021every}  of having at least two ends. We now give the Laplacian separation principle.
\begin{theorem}\label{thm:ConstLaplacian} Let $G=(V,P,d)$ be a reversible metric Markov chain with non-negative Ollivier curvature.
Assume $V=X \dot \cup K \dot \cup Y$ for a finite set $K$ with $E(X,Y) = \emptyset$.
Then, there exists a function $f \in \Lip(1)$ satisfying
\begin{enumerate}[(i)]
\item $\Delta f = C= const.$ on $K$,
\item $f= \min\{g \in \Lip(1): g|_K = f_K\}$ on  $X$, and
\item $f= \max\{g \in \Lip(1): g|_K = f_K\}$ on  $Y$.
\end{enumerate}
Moreover, $\Delta f \geq C$ on $X$ and $\Delta f \leq C$ on $Y$.
\end{theorem}

\begin{proof}
We consider three nested optimization problems for $f \in \Lip(1)$ under the constraint of $(ii)$ and $(iii)$:
\begin{enumerate}[(a)]
\item Maximize $\min_K \Delta f$
\item Minimize $m\left( \argmin \left((\Delta f)|_K\right) \right)$
\item Maximize 
\[\max \{f(y)-f(x): x,y \in K\mbox{ and } \Delta f(x) = \min_K \Delta f \mbox{ and } \Delta f(y)> \min_K \Delta f\}.
\]
\end{enumerate}
Let $F_a \subseteq Lip(1)$ be the set of optimizers of $(a)$.
Let $F_b \subseteq F_a$ be the set of optimizers of $(b)$.
Let $F:=F_c \subseteq F_b$ be the set of optimizers of $(c)$.
A compactness argument shows $F \neq \emptyset$.
Let $f \in F$.
Let
\[
Sf(x) := \begin{cases}
f(x)&: x \in K, \\
\min\{g(x) : g \in \Lip(1), g|_K=f|_K \} &: x \in X, \\ 
\max\{g(x) : g \in \Lip(1), g|_K=f|_K \} &: x \in Y. 
\end{cases}
\]

Let $\eps>0$ and $g := S(f + \eps \Delta f)$.
Indeed, the nested optimization problems are motivated by the question which properties are improved when replacing $f$ by $g$.

By non-negative Ollivier curvature, we have $f + \eps \Delta f \in \Lip(1)$ for small $\eps$ and thus, $g \in Lip(1)$.
As $g$ is in the image of $S$, we see that $g$ satisfies $(ii)$ and $(iii)$. 

Let $C=\min_K \Delta f$.
As $f$ satisfies $(ii)$ and $(iii)$, we have $f=Sf$, and thus, $g \geq f+\eps C$ with equality on $\argmin \left((\Delta f)|_K\right)$.
Hence, $\Delta g \geq C$ on $\argmin \left((\Delta f)|_K\right)$.

For $\eps$ small enough, we have $\Delta g(x) > C$ for all $x \in K \setminus \argmin \left((\Delta f)|_K\right)$. 
As $f$ is an optimizer of $(a)$ and $(b)$, we infer
$\Delta g= C$ on $\argmin \left((\Delta f)|_K\right)$.
This implies
$\argmin \left((\Delta f)|_K\right) = \argmin \left((\Delta g)|_K\right)$, and the maximization in $(iii)$ runs over the same vertex set for $f$ and $g$.

Suppose $\max \{f(y)-f(x): x,y \in K, \Delta f(x) = \min_K \Delta f, \Delta f(y)> \min_K \Delta f\}$ is attained at $x$ and $y$.
Then, 
\[
g(y) - g(x) =  f(y)-f(x) + \eps\Delta f(y) - \eps \Delta f(x) > f(y)-f(x)
\]
contradicting that $f$ maximizes $(c)$. Hence, the maximization in $(c)$ is ill posed, meaning that the maximum is taken over the empty vertex set. This shows that $\Delta f = \min_K \Delta f$ on $K$
as desired. 
We finally prove the 'Moreover' statement, i.e. $\Delta f \geq C$ on $X$ and $\Delta f \leq C$ on $Y$.
Let $x \in X$. As $f$ is the minimal Lipschitz extension on $X$, there exists $y \in K$ such that $f(y)-f(x)=d(x,y)$. As $f \in \Lip(1)$, we get $\Delta f(x) \geq \Delta f(y)=C$. The corresponding estimate on $Y$ can be proven similarly.
This finishes the proof of the theorem.
\end{proof}

We remark that the theorem can be generalized to infinite, locally finite Markov chains with a finite subset $K$.
An interesting question is if there is a natural parabolic flow converging to the solution $f$ from the above theorem.

\subsection{Exploiting the gradient via chain rule}
\label{sec:ChainRule}
As discussed in the introduction, no exact chain rule is available for the graph Laplacian \cite{bauer2015li,erbar2012ricci,munch2014li}.
An approximate chain rule via intermediate values for the graph Laplacian is provided in \cite{hua2022extremal}.
Here, we provide an estimate for $\Delta \phi\circ f$ for suitable functions $\phi:\R \to \R$.
We recall
\[
P_0 = \inf_{x,y} {P(x,y)}\cdot{d(x,y)^2}
\]
and
\[
\nabla_- f (x) = \max_{y\sim x} \frac{(f(y)-f(x))_-}{d(x,y)}.
\]

\begin{lemma}\label{lem:LocalChainRule}
Let $\phi:I \to \R$ be concave.
Let $f \in \R^V$.
Let $x \in V$.
Assume $\phi'''(s) \geq 0$ whenever
\[
\min_{y\sim x} f(y) < s < f(x).
\] 
Then at $x$,
\[
\Delta\phi\circ f \leq \phi'(f)\cdot \Delta f + \frac{P_0}2 \phi''(f) \cdot(\nabla_- f)^2.
\]
\end{lemma}

\begin{proof} Let $x \in V$.
We calculate
\[
\Delta \phi \circ f(x) = \sum_y P(x,y)(\phi(f(y)) - \phi(f(x))).
\]
By Taylor expansion of $\phi$ around $f(x)$, we have
\[
\phi(f(y)) - \phi(f(x)) = \phi'(f(x))(f(y)-f(x)) + \frac 1 2 \phi''(\xi_y) (f(y)-f(x))^2  
\]
where $\xi_y$ is between $f(x)$ and $f(y)$.
First assume $\nabla_-f(x)>0$. Then, there is $z$ such that
\[
f(z)-f(x) = - d(x,y)\nabla_-f(x)
\]
Then, by concavity of $\phi$,
\begin{align*}
\Delta f(x) &\leq \frac 1 2 \phi''(\xi_z)P(x,z)(f(z)-f(x))^2 +\sum_y P(x,y) \phi'(f(x))(f(y)-f(x)) \\
&\leq  \frac 1 2 \phi''(f(x))P(x,z)(f(z)-f(x))^2 +\phi'(f(x))\Delta f(x) \\
&\leq \frac {P_0} 2 \phi''(f)\cdot (\nabla_- f(x))^2 +      \phi'(f(x))\Delta f(x) 
\end{align*}
where the second inequality follows as $f(z)<f(x)$ and $\phi'''\geq 0$, and the last inequality follows as $\phi'' \leq 0$. 
In the case $\nabla_-f(x)=0$, we ignore the $\phi''(\zeta_y)$ term and get the desired estimate easily. This finishes the proof.
\end{proof}

We now show how to improve a Laplacian estimate via the chain rule.
\begin{lemma}\label{lem:GlobalChainRule}
Let $W \subset V$. Let $u \in \R^V$ with $\Delta u \leq C$ on $W$ and $\nabla_-f \geq 1$ on $W$.
We write $\beta = 2C/P_0$.
Let
\[
U_0 := \{x:u(y) \geq 0 \mbox{ for all } y \sim x\}.
\]
Let $R>0$. Then, there exists an increasing, concave $\phi:\R \to \R$ with $\phi'\leq 1$ such that for all $x \in W$,
\[
\Delta \phi \circ u(x) \leq  \begin{cases} 0
&: u(x)>R\\
C&:x \in W \setminus U_0\\
C \wedge \frac{-P_0}{2R}&: x \in U_0 \mbox{ and } u(x)\leq R \mbox{ and } C\leq 0 \\
-\frac{C}{\exp({\beta R})-1}&: x \in U_0 \mbox{ and } u(x)\leq R \mbox{ and } C> 0. 
\end{cases} 
\]
Moreover, the last two cases can be unified as
\[
\Delta \phi \circ u \leq - \frac{C/2}{\exp({\beta R})-1}.
\]

\end{lemma}
\begin{proof}
We first assume $C> 0$. Then for concave, increasing $\phi$ with $\phi''' \geq 0$, we have on $W$
\[
\Delta \phi \circ f \leq C \phi'(f) +  \frac {P_0}2 \phi''(f).
\]
Let $\beta = 2C/P_0$.
On $[0,R]$, we define
\[
\phi(s) = \frac{1- \exp(-\beta s) - \beta s \exp(-\beta R)}{\beta(1-\exp(-\beta R))}
\]
giving
\[
C\phi' + \frac {P_0}2 \phi'' = \frac{-C}{ \exp(\beta R)-1}.
\]
We extend $\phi$ linearly outside the interval $[0,R]$, i.e., $\phi(s)=-s$ for $s<0$ as $\phi(0)=0$ and $\phi'(0)=-1$. Moreover, we have $\phi'(R)=0$, and thus, we set $\phi(s)=\phi(R)$ for $s>R$.
Applying Lemma~\ref{lem:LocalChainRule}, the claim follows easily in the case $C>0$.

We now consider the case $\frac {- P_0}{2R} \leq C \leq  0$.
For $s \in [0,R]$, we define
\[
\phi(s) =  \frac{2Rs - s^2}{2R} 
\]
giving
\[
C\phi' + \frac {P_0}2 \phi'' = \frac{C(2R - 2s) -P_0}{2R} \leq \frac  {-P_0}{2R} 
\]
As above, we have $\phi'(0)=1$ and $\phi'(R)=0$, and the conclusion follows similarly.
Finally, in the case $C<\frac {- P_0}{2R}$, we use $\phi(s)=s$ and obtain on $W$
\[
\Delta \phi \circ u = \Delta u \leq C  
\]
where we used $\beta R<-1$.
This finishes the case distinction.
The 'Moreover' statement is a straight forward calculation and thus, the proof is finished.
\end{proof}

\subsection{Log Sobolev and Dirichlet eigenvalue estimate}
We now use the Laplacian separation principle and the chain rule provided above to prove a lower bound for $\as$. We recall that $c\alpha \leq \as \leq C\alpha$ for universal constants $c,C>0$, see \ref{cor:asalpha}. Also recall that
\[
\as = \inf_{W\subset V} \theta(\lambda_W,\lambda_{W^c})
\]
where $\theta$ is the logarithmic mean.
\begin{theorem}\label{thm:asDiam}
Let $G=(V,P,d)$ be a reversible metric Markov chain with non-negative Ollivier Ricci curvature. Then,
\[
\as \geq \frac {P_0} {16 \diam(G)^2}.
\]
\end{theorem}
\begin{proof}
Let $V=A \dot\cup B$. We aim to give a lower bound to the logarithmic mean $\theta(\lambda_A,\lambda_B)$ where $\lambda_A$ and $\lambda_B$ are the corresponding Dirichlet eigenvalues.
We apply Theorem~\ref{thm:ConstLaplacian} to $K=\supp(\Delta 1_A)$ and $X=A\setminus K$ and $Y=B\setminus K$ to obtain a function $f$ with $\Delta f = C$ on $K$.
As $f \in \Lip(1)$, we can assume $\diam(G) \leq f \leq 2 \diam (G)$.
Let $g=f \cdot 1_A$.
Then, $\nabla_- g \geq 1$ on $B$ and $\Delta g \leq C$ on $A$. By Lemma~\ref{lem:GlobalChainRule} applied to $W=B$, and $u=g$ and $R=2\diam(G)$ we get for all $y \in B$
\[
\Delta \phi \circ u (y)  \leq -\frac{C/2}{\exp(\beta R)-1}
\]
where $\beta = 2C/P_0$.
Moreover $0 \leq \phi \circ u \leq R$ as $\phi' \leq 1$.
Hence by Lemma~\ref{lem:DirichletSupersolution},
\[
\lambda_B \geq \frac 1 R \cdot \frac{C/2}{\exp(\beta R)-1} = \frac {P_0} {4R^2} \cdot \frac{\beta R}{\exp(\beta R)-1}.
\]
By a similar argument,
\[
\lambda_A \geq \frac {P_0} {4R^2} \cdot \frac{-\beta R}{\exp(-\beta R)-1} = \frac {P_0} {4R^2} \cdot e^{\beta R}\cdot\frac{\beta R}{\exp(\beta R)-1}.
\]
We recall $\theta(s,t)=(s-t)/(\log(s)-\log(t))$ and thus,
\[
\theta(\lambda_A,\lambda_B) \geq \frac {P_0} {4R^2} = \frac {P_0} {16\diam(G)^2}.
\]
As $A,B$ were chosen as an arbitrary partition of $V$, the claim of the theorem follows immediately.
\end{proof}

As $\as$ and $\alpha$ coincide up to a bounded factor (see Corollary~\ref{cor:asalpha}), we obtain the following corollary
\begin{corollary}
Let $(V,P,d)$ be a reversible metric Markov chain. Then,
\[
\alpha \geq \frac{CP_0}{\diam^2}
\]
for a universal constant $C$.
\end{corollary}

\section{Non-negative curvature and isoperimetry}\label{sec:RelatedDiameter}

By the pioneering work of Ledoux \cite{ledoux2011concentration} and Milman \cite{milman2009role,milman2010Isoperimetric} on weighted manifolds with non-negative Ricci curvature, various concentration of measure inequalities imply various isoperimetric inequalities. 
This goes under the name 'Reversing the hierarchy' as in a general smooth setting one schematically has the following implications:
\begin{align*}
\mbox{Isoperimetric inequalities} &\Rightarrow \mbox{Functional inequalities} \\&\Rightarrow \mbox{Transport entropy inequalites}\\& \Rightarrow  \mbox{Concentration inequalities},
\end{align*}
see \cite{milman2012properties}. While in general the reverse directions are wrong, they still hold true in case of non-negative Ricci curvature.
Here, we give discrete versions of the reversal of hierarchy.
By doing so, we give an affirmative answer to \cite[Conjecture~6.9]{erbar2018poincare} and \cite[Conjecture~6.10]{erbar2018poincare} in case of non-negative Ollivier curvature.
Moreover our isoperimetric inequalities seem the first ones which are explicit and only require an arbitrary single point from the concentration profile.

\subsection{A conjecture by Erbar and Fathi}\label{sec:ErbarFathiConjecture}

We recall \cite[Conjecture~6.9]{erbar2018poincare} which we answer affirmatively for graphs with non-negative Ollivier Ricci curvature in this section.
Assume we have concentration profile $Me^{-\beta r}$, i.e.,
\[
m(A_r^c) \leq Me^{-\beta r}
\]
whenever $m(A)\geq \frac 1 2$ where 
$A_r = \{x: d(x,A)> r\}.$
It was conjectured in \cite[Conjecture~6.9]{erbar2018poincare} that this implies
$\lambda \geq C(M) \beta^2$.

To prove this conjecture, we use a notion of observable diameter $\diamess$ capturing the concentration of measure.
The observable diameter  is defined as 
\begin{align}
\diamess^{(\eps)} := \sup_{\substack{m(A)\geq \eps\\m(\cl(B)) \geq \eps}} d(A,B)\label{eq:ObsDiam}
\end{align}

where $\cl$ denotes the closure, i.e. it adds the outer vertex boundary. The reason behind taking closure is that if a large fraction of the mass is located at a single vertex, then $A$ and $B$ would have to overlap so that their distance would be zero.
The key step to prove the conjecture is to show for non-negatively curved graphs,
\[
h \simeq \frac{1}{\diamess^{(1/8)}}
\]
where $h$ is the Cheeger constant.
The precise statement is given in Corollary~\ref{cor:hlargerdiamess} and Theorem~\ref{thm:hsmallerdiamess}.
We then show in Lemma~\ref{lem:concentrationDiamess} that
\[
\diamess \lesssim \frac 1 \beta
\]
and the conjecture follows easily by the Cheeger-Buser inequality
\[
\lambda \simeq h^2,
\]
for for graphs with non-negative Ollivier Ricci curvature, see \cite[Theorem~3.2.2]{munch2019non} for the Buser inequality $\lambda \gtrsim h^2$.

\subsection{Isoperimetry and observable diameter}

To motivate our new isoperimetric inequality, we
recall a result from \cite{munch2019non} stating that for non-negatively curved graphs,
\[
h \geq \frac{P_0}{4 \diam(G)},
\]
see \cite[Theorem~3.5.1 and 3.2.2]{munch2019non}.
We notice that $h$ is not sensitive to narrow tails, but the diameter is. 
This discrepancy can be seen as reason why the above estimate does not give precise bounds in case of high dimension, as high dimension generally comes with narrow tails. 

This gives motivation to use a notion of observable diameter (see \eqref{eq:ObsDiam}) which can be seen as a diameter like quantity enforced to be insensitive to narrow tails. 
In the smooth setting it is well known that the observable diameter as function in $\eps$ is, up to constant factors, precisely the inverse function of the concentration profile, see e.g. \cite{ozawa2015estimate}. 
We give a discrete version in Section~\ref{sec:ObsDiam}.
For our isoperimetric inequality, we recall the boundary measure
\[
|\partial W| = -\mathcal E(1_W,d(W,\cdot)) = \sum_{\substack{x \in W\\y\notin W}} P(x,y)m(x)d(x,y),
\]
the minimum transition rate
\[
P_0 = \inf_{x\sim y}P(x,y)d(x,y)^2
\]
and the observable diameter
\[
\diamess^{(\eps)} = \max_{\substack{m(A)\geq \eps \\ m(\cl(B)) \geq \eps}} d(A,B).
\]

We now give our main isoperimetric inequality in terms of the observable diameter. 
We remark that the well known method of using gradient estimates for the heat semigroup seem to be not powerful enough to prove this estimate in a discrete setup due to the lack of a discrete chain rule. Instead, we will use the Laplacian separation principle with which we already proved the log-Sobolev inequality. 
Another feature distinguishing our result from the smooth setting is that we allow to choose the Markov operator $P$ and the distance function independently, while in the smooth setting, the underlying Laplace operator uniquely determines the distance function.

\begin{theorem}\label{thm:IsoperimetryConcentration}
Let $G=(V,P,d)$ be  a reversible metric Markov chain with non-negative Ollivier curvature.
Let $W \subset V$ with $m(W)\leq \frac 1 2$. Let $\eps \leq \frac 1 8$.
Then,
\[
|\partial W| \geq P_0\cdot \frac{(m(W) \wedge \eps) \log (1/\eps) }{6\diamess^{(\eps)}}.
\]
\end{theorem}

\begin{proof}
Let $K$ be the inner vertex boundary of $W$. Let $X=W\setminus K$ and $Y=V\setminus W$.
By the Laplacian separation principle (Theorem~\ref{thm:ConstLaplacian}), there exists a function $f \in \Lip(1)$ with $\Delta f = C$ on $K$.
Let $R:=\diamess^{(\eps)}$. We proceed with case distinction with respect to $C$.

We first assume \[C \geq \frac{P_0 \log(1/\eps)}{6R}.\]
Then,
\[
|\partial W| \geq \mathcal E(1_W,-f) \geq Cm(W),
\]
and the claim follows easily in this case.

Now assume 
\[C < C_0:= \frac{P_0 \log(1/\eps)}{6R}.\]

We choose $r \in \R$ maximal so that $m(\{f \geq R+r\}) \geq \eps$.
As $f \in \Lip(1)$ and by maximality of $r$,
\[
m(\{f>R+r\})\leq  \eps \quad\mbox{ and } \quad m(\cl(\{f<r\})) \leq \eps.
\]
Without loss of generality, we assume $r=0$.
By Lemma~\ref{lem:GlobalChainRule} applied to $V\setminus W$, there exists $g = \phi \circ f \in \Lip(1)$ such that on $V \setminus W$,
\[
\Delta g(x) \leq \begin{cases}
0 &: f(x)>R, \\
C_0 &: x \in \cl(\{f<0\}), \\
-\frac{C_0}{\exp(\beta R)-1}&: \mbox{ else},
\end{cases}
\]
where $\beta =2C_0/P_0$.
We write $U=V\setminus W$ and
$U_R = U \cap \{f>R\}$ and $U_0 = U \cap \cl(\{f<0\})$ and $\Uint = U \setminus U_R \setminus U_0$.
Then,
\begin{align*}
|\partial W| \geq \mathcal E(1_{U},g) &\geq \frac{C_0}{\exp(\beta R) - 1} m(\Uint) - C_0m(U_0) \\
&\geq \frac{C_0}{\exp(\beta R) - 1} (m(U)-2\eps) - C_0\eps\\
&\geq \frac{C_0/4}{\exp(\beta R) - 1}  -  C_0\eps
\end{align*}
where we used $\eps\leq 1/8$ and $m(U) \geq \frac 1 2$ in the last estimate.
As $C_0 = \frac{P_0 \log(1/\eps)}{6R}$, we get $\exp(\beta R)  \leq \eps^{-1/6}$ and thus,
\begin{align*}
|\partial W| \geq C_0 \left( \frac{1/4}{\eps^{-1/3} - 1} - \eps \right) \geq C_0 \eps
\end{align*}
where we used $\eps \leq 1/8$. Now the claim follows easily which finishes the case distinction and thus the proof.
\end{proof}

It turns out that the case $\eps<m(W)$ does not give any improvement over the choice $\eps=m(W)$, so that we easily obtain the following corollary.
\begin{corollary}\label{cor:IsoperimetryConcentrationLargeEps}
Let $G=(V,P,d)$ be  a reversible metric Markov chain with non-negative Ollivier curvature.
Let $W \subset V$ with $m(W)\leq \frac 1 2$. Let $\eps \in \left[\frac{m(W)}{4}, \frac 1 8 \right]$.
Then,
\[
\frac{|\partial W|}{m(W)} \geq P_0\cdot \frac{\log (1/\eps) }{24\diamess^{(\eps)}}.
\]
\end{corollary}

Plugging in $\eps = 1/8$ gives the following estimate for the Cheeger constant.
\begin{corollary}\label{cor:hlargerdiamess}
For graphs with non-negative Ollivier curvature,
\[
h \geq \frac {P_0}{12  \diamess^{(1/8)} }.
\]
\end{corollary}

A key advantage of using the observable diameter is that we also get the reverse inequality for the Cheeger constant.
For convenience, we restrict ourselves here to the case of the combinatorial distance.
We recall
\[
\D = \max_{x \in V} \Delta d(x,\cdot)(x) = \max_{x\in V} \sum_{y\neq x} P(x,y).
\]

\begin{theorem}\label{thm:hsmallerdiamess}
Let $G=(V,P,d)$ be a reversible metric Markov chain with combinatorial distance and $m(V)=1$.
Then,
\[
h \leq \frac{57\D}{\diamess^{(1/8)}}.
\]
\end{theorem}
\begin{proof}
We fist assume $\diamess^{(1/8)} >1$.
Let $A,B$ maximize the expression for $\diamess^{(1/8)}$.
We consider $f:=d(A,\cdot)$.
Then,
\[
\frac 1 8 (\diamess^{(1/8)} -1) \leq (\diamess^{(1/8)} -1)m(\cl(B)) \leq \|f\|_1 = \int_0^\infty m(f>r)dr
\]
Moreover,
\begin{align*}
\D \geq \frac 1 2\sum_{x,y}P(x,y)m(x)|f(y)-f(x)| =\|\nabla f\|_1 =\int_0^\infty |\partial (\{f>r\})|dr
\end{align*}
Hence, there exists $r_0>0$ such that
\[
\frac{|\partial (\{f>r_0\})|}{m(f>r_0)} \leq \frac{8\D}{(\diamess -1)}
\]
As $m(f>r) \leq \frac 7 8$ for all $r>0$, we get
\[
m(f>r) \leq \frac 1 7 m(f \leq r)
\]
giving
\[
h \leq \frac{|\partial (\{f>r_0\})|}{m(f>r_0) \wedge m(f\leq r_0)} \leq \frac{56\D}{(\diamess^{(1/8)} -1)}.
\]
On the other hand, we generally have $h \leq \D$ by choosing a single vertex set. Combining with the estimate above, the claim of the theorem follows easily.
\end{proof}

We remark that the upper and lower bound for $h$ from Corollary~\ref{cor:hlargerdiamess} and Theorem~\ref{thm:hsmallerdiamess} differ by a factor of order $\D/P_0$. The same phenomenon can be seen for the discrete Cheeger-Buser inequality and can be explained by the use of qualitatively different gradient notions for both sides of the estimate.


\subsection{Concentration of measure versus observable diameter} \label{sec:ObsDiam}

The essntial diameter is closely related to the concentration of measure phenomenon.
For a set $A \subset V$, we define
\[
A_r = \{x: d(x,A)\leq  r\}.
\]

In a slightly different setup, the following lemma was shown in \cite[Proposition~2.6]{ozawa2015estimate} and references therein.

\begin{lemma}\label{lem:concentrationDiamess}
Let $G=(V,P,\dc)$ be a reversible metric Markov chain with combinatorial distance $\dc$. Let $r>0$.
Assume at distance $r \in \N$, we have concentration $\eps$, i.e.,\[
m(A_r^c) < \eps
\]  whenever $m(A)\geq \frac 1 2$. Then, 
\[
\diamess^{(\eps)} \leq 2r.
\]
\end{lemma}

\begin{proof}
For simplicity, we write $d=\dc$.
We will throughout use that
\[
A \subseteq (A_R^c)_R^c
\]
and
\[
R+d(A_R,B) \geq  d(A,B)
\]
for all $R>0$ and $A,B \subseteq V$.
Let $A,B$ attain the observable diameter.
We notice $m(A_r)> \frac 1 2 $ by the concentration assumption and as $m(A)\geq \eps$.
Let $R \in \N$ such that $m(A_R)\leq \frac 1 2$ and $m(A_{R+1}) \geq \frac 1 2$.
Then, $R < r$.
As $m(\cl(B)) \geq \eps$, we similarly get
\[
d(A_{R+1},\cl(B))<r.
\]
Thus,
\[
d(A,B) \leq 1 + d(A,\cl(B)) \leq 1+ R+ 1 + d(A_{R+1},\cl(B)) \leq r+R-1 \leq 2r,
\]
finishing the proof.
\end{proof}

\subsection{Gaussian concentration implies Gaussian isoperimetry}

A specific application of 'Reversing the hierarchy' is that on weighted manifolds with non-negative Ricci curvature,
 Gaussian concentration of measure implies Gaussian isoperimetry
 \cite{ledoux2011concentration,milman2009role,
 milman2010Isoperimetric}  
  which is equivalent to a log-Sobolev inequality, see e.g. \cite{ledoux2006isoperimetry}.
Here, we give a discrete version of this result using our general isoperimetric inequality from Corollary~\ref{cor:IsoperimetryConcentrationLargeEps}.
By doing so, we give an affirmative answer to \cite[Conjecture~6.10]{erbar2018poincare} in case of non-negative Ollivier curvature. We now give the details.
The Gaussian isoperimetric constant is defined  as
\[
h_{\sqrt {\log}} = \inf_{m(W)\leq \frac 1 2} \frac{|\partial W|}{m(W)\sqrt{\log(1/m(W))}}
\]

It is shown in
\cite[Remark~5]{houdr2001mixed}
that
\[
\alpha \gtrsim \frac{h_{\sqrt {\log}}^2}{\D}
\]
and in \cite{klartag2015discrete} that in case of non-negative Bakry-Emery curvature,
\[
\alpha \lesssim \frac{ h_{\sqrt {\log}}^2}{P_0}.
\]
As the gradient estimate 
\[
\|\nabla P_t f\|_\infty \lesssim \frac{\|f\|_\infty}{\sqrt{P_0 t}}
\]
is the only consequence of non-negative Bakry Emery curvature needed in the proof, and the same gradient estimate also holds true in case of non-negative Ollivier Ricci curvature, it is expected that the upper bound for $\alpha$ is also valid in case of non-negative Ollivier curvature, so that, up to a factor $P_0$, we have
\[
\alpha \simeq h_{\sqrt {\log}}^2.
\]

\begin{theorem}
Let $G=(V,P,\dc)$ be a reversible metric Markov chain with non-negative Ollivier curvature and combinatorial distance $\dc$.
Assume that for all $A$ with $m(A) \geq \frac 1 2$,
\[
m(A_r) \leq \exp(-\rho r^2).
\]
Then,
\[
h_{\sqrt {\log}} \geq \frac{P_0}{48}\sqrt{\rho}.
\]

\end{theorem}

\begin{proof}
Without loss of generality, we can assume that the concentration inequality is strict, otherwise, we could replace $\rho$ by $\rho - \eps$, and at the end take the limit $\eps \to 0$.
By Lemma~\ref{lem:concentrationDiamess}, we have
\[
\diamess^{(\eps)} \leq 2 \sqrt{\frac{\log(1/\eps)}{\rho}}.
\]
Let $W \subset V$ with $m(W) \leq \frac 1 2$.
By Corollary~\ref{cor:IsoperimetryConcentrationLargeEps}  with $\eps = m(W)/4$,
\[
\frac{|\partial W|}{P_0 m(W)} \geq \frac{\log (1/\eps)}{24\diamess^{(\eps)}} \geq \frac{\sqrt{\rho}}{48}  \sqrt{\log(1/\eps)} \geq  \frac{\sqrt{\rho}}{48}  \sqrt{\log(1/m(W))}
\]
and the claim follows by rearranging.
\end{proof}

As a corollary, we obtain that Gaussian concentration implies a log-Sobolev inequality for non-negatively curved Markov chains. 

\begin{corollary}
Let $G=(V,P,\dc)$ be a reversible metric Markov chain with non-negative Ollivier curvature and combinatorial distance $\dc$.
Assume that for all $A$ with $m(A) \geq \frac 1 2$ and all $r \in \N$,
\[
m(A_r) \leq \exp(-\rho r^2).
\]
Then, 
\[
\alpha \geq CP_0^2 \rho
\]
for some universal constant $C$.
\end{corollary}
We remark that the dependence on the minimal transition rate $P_0$ is not avoidable, as Markov chains with curvature at least $K>0$ admit Gaussian concentration with $\rho = K$, see \cite[Theorem~3.1]{jost2019Liouville}. However, the log-Soblev constant cannot be lower bounded purely in terms of the curvature, as we showed in the example in Section~\ref{sec:CounterExample}.

\subsection{Isoperimetry and diameter of subsets}

So far, we gave isoperimetric estimates in terms of the observable diameter of the whole space. In this subsection, we give isoperimetric estimates in terms of the diameter of the subset.
For a subset $W \subset V$, we write
\[
\diam(W) = \max_{x,y \in W} d(x,y)
\]
where $d$ is the original graph distance (and not the distance within the induced subgraph).

\begin{theorem}\label{thm:internalDiameter}
Let $G=(V,P,d)$ be a metric Markov chain graph with non-negative Ollivier curvature. Let $W \subset V.$ Then,
\[
|\partial W| \geq \frac{P_0}{\diam(\cl(W))} m(W)(1-m(W)).
\]
\end{theorem}

\begin{proof}
We apply the Laplacian separation principle (Theorem~\ref{thm:ConstLaplacian}) to $Y=W$ and $K=\supp 1_W  \setminus W$ and $X=V\setminus W \setminus K$.
We obtain a function $f$ with $\Delta f = C$ on $K$.
Let $R=\cl(W) = Y \cup K$
We can assume that $0 \leq f \leq R$ on $Y \cup K$.

We first assume $C>0$.
By Lemma~\ref{lem:GlobalChainRule}, we get
\[
\Delta g \leq -\frac{C}{\exp(\beta R) - 1} \quad \mbox{ on } Y
\]
with $\beta = 2C/P_0$ and some $g \in \Lip(1)$.
Hence,
\[
|\partial W| \geq \mathcal E(1_W,g) \geq m(W)\frac{C}{\exp(\beta R) - 1}.
\]
On the other hand,
\[
|\partial W| \geq \mathcal E(1_{V\setminus W},-f) \geq C(1-m(W))
\]
Thus,
\begin{align*}
\frac 1 {\partial W} \cdot 
\frac{m(W)(1-m(W))}{m(W)+(1-m(W))} \leq  \frac{\frac{\exp(\beta R) - 1}{C}\cdot \frac 1 C}{\frac{\exp(\beta R) - 1}{C}+ \frac 1 C} = \frac{1-\exp(-\beta R)}{C}
\leq \frac {2R}{P_0}.
\end{align*}
Rearranging gives
\[
|\partial W| \geq \frac{P_0}{2R}m(W)(1-m(W)).
\]
In the case $C\leq 0$, we again use Lemma~\ref{lem:GlobalChainRule} to get $g \in \Lip(1)$ with
\[
\Delta g \leq - \frac{P_0}{2R} \mbox{ on } W
\]
and hence,
\[
|\partial W| \geq \frac{P_0}{2R} m(W).
\]
and the desired estimate follows easily.
This finishes the case distinction and thus the proof.
\end{proof}

\section*{Acknowledgments}
The author wants to thank Yong Lin for pointing out the question if the Dirichlet eigenvalue on balls of radius $R$ of infinite graphs is at least $1/R^2$. This question turned out to be the missing puzzle piece for finding the proof of the isoperimetric inequalities. The author also wants to thank Emanuel Milman for fruitful discussions.

\printbibliography

	Florentin Münch, \\
	MPI MiS Leipzig, 04103 Leipzig, Germany\\
	\texttt{florentin.muench@mis.mpg.de}\\

\end{document}